\documentclass[12pt]{article}
\usepackage{amsmath,amsfonts,amssymb}

\def\A{$====\Rightarrow$}
\def\B{$\Leftarrow====$}
\begin{document}

\newcommand{\tr}{\mathop{\rm tr}}
\newcommand{\diag}{\mathop{\rm diag}}
\newenvironment{proof}
{\noindent{\bf Proof:}}{\qed}

\newcommand{\qed}{
{\vrule height8pt width5pt depth0pt} \vskip 6pt }
\newcommand{\op}{ \mathop{ \vbox{\hrule\hbox{\vrule
height7pt\kern5pt\vrule}\hrule}\thinspace } }

\protect\newcounter{footercount}[page]
\def\symbolfootnote#1{
\addtocounter{footercount}{1}
\def\thefootnote{\fnsymbol{footercount}} \footnote{#1} }

\newtheorem{theorem}{Theorem}
\newtheorem{proposition}[theorem]{Proposition}
\newtheorem{con}[theorem]{Conjecture}
\newtheorem{remark}[theorem]{Remark}
\newtheorem{lemma}[theorem]{Lemma}
\newtheorem{definition}[theorem]{Definition}
\newtheorem{cor}[theorem]{Corollary}
\newtheorem{example}[theorem]{Example}
\font\german=eufm10
\def\a{{\hbox{\german a}}}
\def\g{{\hbox{\german g}}}
\def\h{{\hbox{\german h}}}
\def\k{{\hbox{\german k}}}
\def\m{{\hbox{\german m}}}
\def\n{{\hbox{\german n}}}
\def\p{{\hbox{\german p}}}
\def\so{{\hbox{\german so}}}
\def\sopq{{\scriptscriptstyle {\bf SO}(p,q)}}
\def\sl{{\scriptscriptstyle {\bf SL}(p,\R)}}
\def\R{{\bf R}}
\def\C{{\bf C}}
\def\F{{\bf F}}
\def\H{{\bf H}}
\def\SO{{\bf SO}}
\def\SU{{\bf SU}}
\def\SL{{\bf SL}}
\def\sln{\hbox{\goth sl}}

\def\S{\mathcal{S}}

\def\Ad{\mathop{\hbox{Ad}}}
\def\ad{\mathop{\hbox{ad}}}

\begin{center}
{\Large \bf On the product formula on non-compact Grassmannians} \\
P. Graczyk and P.\ Sawyer\symbolfootnote{\noindent Supported by l'Agence Nationale de la Recherche
ANR-09-BLAN-0084-01.}
\end{center}

\begin{abstract}
We study the absolute continuity of the convolution
$\delta_{e^X}^\natural \star \delta_{e^Y}^\natural$ of two orbital measures on 
the symmetric space ${\bf SO}_0(p,q)/{\bf SO}(p)\times{\bf SO}(q)$, $q>p$.
We prove sharp conditions on $X$, $Y\in\a$ for the
existence of the density of the convolution measure. This measure intervenes in the
product formula for the spherical functions. We show that the sharp criterion developed for $\SO_0(p,q)/\SO(p)\times\SO(q)$ will also serve for the spaces ${\bf SU}(p,q)/{\bf S}({\bf U}(p)\times{\bf U}(q))$ and ${\bf Sp}(p,q)/{\bf Sp}(p)\times{\bf Sp}(q)$, $q>p$. We also apply our results to the study of absolute continuity of convolution powers
of an orbital measure $\delta_{e^X}^\natural$.
\end{abstract}

\section{Introduction}

The spaces ${\bf SO}_0(p,q)/{\bf SO}(p)\times {\bf SO}(q)$ where $q>p$ (which we will assume throughout the paper), are the noncompact duals of real Grassmannians. 
They are Riemannian symmetric spaces of noncompact type corresponding
to root systems of type \A$B_p$\B. The harmonic analysis
on these spaces has been intensely developed in recent years (\cite{Anker,Rosler,Sawyer1,Sawyer2}).

We use throughout the paper the usual notations of the harmonic analysis on Riemannian symmetric spaces.
The books \cite{Helga0, Helga} constitute a standard reference on these topics. 

Let $X$, $Y\in \a$ and let $m_K$ denote the Haar measure of the group $K$. We define
 $\delta_{e^X}^\natural=m_K\star \delta_{e^X} \star m_K$.
The question of the absolute continuity of the convolution $\delta_{e^X}^\natural \star
\delta_{e^Y}^\natural$ of two $K$-invariant orbital measures that we address in our paper
has important applications in harmonic analysis itself (the product formula for the
spherical functions) and in probability theory (random walks, $I_0$
characterization of Gaussian measures).

The spherical Fourier transform of the measure $\delta_{e^X}^\natural$ is equal
to the spherical function $\phi_\lambda (e^X)$, where 
 $\lambda$ is a complex-valued linear form on $\a$. Thus the product 
 $\phi_\lambda (e^X)\phi_\lambda (e^Y)$
 is the spherical Fourier transform of the convolution $m_{X,Y}=\delta_{e^X}^\natural \star
\delta_{e^Y}^\natural$. If we denote by $\mu_{X,Y}$ 
the projection of the measure $m_{X,Y}$ on $\a$ via the Cartan decomposition $G=KAK$, then

\begin{align*}
\phi_\lambda (e^X)\,\phi_\lambda (e^Y) =\int_{\a}\,\phi_\lambda (e^H)\,d\mu_{X,Y} (H).
\end{align*}
Let $\delta$ be the density of the invariant measure on $\a$ in polar coordinates.
The existence of a kernel in the last product formula 
\begin{align}\label{kernel}
\phi_\lambda(e^X)\,\phi_\lambda(e^Y) 
=\int_{\a^+}\,\phi_\lambda(e^H)\,k(H,X,Y)\,\delta(H)\,dH 
\end{align}
 is equivalent to the absolute continuity of the measure $\mu_{X,Y}$ with respect to the Lebesgue measure on $\a$
 and to the existence of the density of $m_{X,Y}$ on $G$, with respect to the invariant measure $dg$.
 When the formula (\ref{kernel}) holds, we say that we have a product formula for $X$ and $Y\in\a$.
 Provided that $X$, $Y\in \a^+$,
 the product formula (\ref{kernel}) has been
shown previously (see \cite{Flensted} in the rank one case, \cite{PGPS1} in the complex case and
\cite{PGPS2} in the general case). In \cite{PGPS2} we were able to relax these conditions and show that $\mu_{X,Y}$ is absolutely continuous provided one of $X$ or $Y$ is in $\a^+$ as long as the other is nonzero. The density can however exist in some cases when both $X$ and $Y$ are singular. It is a challenging problem to characterize all such pairs $X$ and $Y$. 

This problem was solved in \cite{PGPS_Lie2010} for symmetric spaces with root system of type $A_n$. We solve it in this paper for the space ${\bf SO}_0(p,q)/{\bf SO}(p)\times {\bf SO}(q)$: we give a definition of an eligible pair $(X,Y)$ (Definition \ref{defEligible}) and next we prove the necessity (Proposition \ref{Necessary}) and the sufficiency (Proposition \ref{VXVYp} and Theorem \ref{MainReduced}) of this property for the absolute continuity of
$m_{X,Y}$.

\A
By \cite{PGPS1,PGPS2}, the density $k(H,X,Y)$ exists if and only if $\S_{X,Y}=a(e^X\,K\,e^Y)$,
the support of the measure $\mu_{X,Y}\big|_{\overline {\a^+}}$, has nonempty interior. Similarly, the density of the measure $m_{X,Y}$ exists if and only if its support $Ke^X Ke^Y K$ has nonempty interior as seen in \cite{PGPS_Lie2010}. These facts are crucial in the proofs of the results of this paper. 

We show in Corollary \ref{iff} that the result for the space ${\bf SO}_0(p,q)/{\bf SO}(p)\times {\bf SO}(q)$ also implies the result for the spaces ${\bf SU}(p,q)/{\bf S}({\bf U}(p)\times {\bf U}(q))$ and ${\bf Sp}(p,q)/{\bf Sp}(p)\times {\bf Sp}(q)$. We conclude the paper with two further applications of our main result. One of them is a characterization of an optimal convolution power $l$ of the measure $\delta_{e^X}^\natural$, which is absolutely continuous for any $X\not=0$, $X\in\a$. Theorem \ref{*power} solves on non-compact Grassmannians a problem raised by Ragozin in \cite{Ragozin}.
\B

\section{Basic properties}\label{root}

We start by reviewing some useful information on the Lie group
${\bf SO}_0(p,q)$, its Lie algebra $\so(p,q)$ and the corresponding
root system. Most of this material was given in \cite{Sawyer2}. 
For the convenience of the reader, we gather below the properties
we will need in the sequel.

In this paper, $E_{ij}$ is a rectangular matrix with 0's everywhere except at at the position $(i,j)$ where it is 1.

Recall that ${\bf SO}(p,q)$ is the group of matrices $g\in{\bf
SL}(p+q,\R)$ such that $g^T\,I_{p,q}\,g=I_{p,q}$
where 
$
I_{p,q}= \left [
\begin{array}{cc}
-I_p&0_{p\times q}\\ 
0_{q\times p}&I_q 
\end{array}
\right]$.
Unless otherwise specified, all $2\times 2$ block
decompositions in this paper follow the same pattern.

The group ${\bf SO}_0(p,q)$ is the connected component of ${\bf
SO}(p,q)$ containing the identity.
The Lie algebra $\so(p,q)$ of ${\bf SO}_0(p,q)$ consists of the
matrices 
\begin{align*}
\left [\begin{array}{cc}A&B\\ B^T&D
\end{array}\right]
\end{align*}
where $A$ and $D$ are skew-symmetric.

 A very important element in our investigations is the Cartan
decomposition of $\so(p,q)$ and ${\bf SO}(p,q)$.
The maximal compact subgroup $K$ is the subgroup of ${\bf SO}(p,q)$
consisting of the matrices
\begin{align*}
\left [\begin{array}{rr} 
A&0\\ 
0&D \end{array}\right]
\end{align*}
of size $(p+q)\times (p+q)$ such that $A\in
{\bf SO}(p)$ and $D\in {\bf SO}(q)$ (hence $K \simeq {\bf SO}(p)
\times {\bf SO}(q)$). If $\k$ is the Lie algebra of $K$ and $\p$ is
the set of matrices 
\begin{align*}
\left [\begin{array}{rr} 
0&B\\ 
B^T&0 \end{array}\right]
\end{align*}
then the Cartan decomposition is given by
$\so(p,q)=\k\oplus\p$ with corresponding Cartan involution
$\theta(X)=-X^T$.\\

The Cartan space $\a\subset \p$ is the set of matrices 
\begin{align*}H=\left
[\begin{array}{ccc}
0_{p\times p}&{\cal D}_H&0_{p\times (q-p)}\\ 
{\cal D}_H&0_{p\times p}&0_{p\times(q-p)}\\
0_{(q-p)\times p}&0_{(q-p)\times p}&0_{(q-p)\times (q-p)}
\end{array}\right]
\end{align*}
where ${\cal D}_H=\diag[H_1,\dots,H_p]$. Its canonical basis is given
by the matrices
\begin{align*}
A_i:=E_{i,p+i} + E_{p+i,i}, 1\leq i\leq p.
\end{align*} 

The restricted roots and associated root vectors for the Lie algebra
$\so(p,q)$ with respect to $\a$ are given in Table \ref{X}.
\begin{table}[h]
\begin{center}
\begin{tabular}{|c|c|c|}\hline
root $\alpha$&multiplicity&root vectors $X_\alpha$\\ \hline
$\alpha(H)=\pm H_i$&$q-p$&
$X_{ir}^\pm=E_{i\,2\,p+r}+E_{2\,p+r\,i}\pm(E_{p+i\,2p+r}-E_{2p+r\,p+i})$\\ 
$1\leq i\leq p$&&$r=1$, \dots, $q-p$\\ \hline
$\alpha(H)=\pm(H_i-H_j)$&1&
$Y_{ij}^\pm=\pm(E_{ij}-E_{ji}+E_{p+i\,p+j}-E_{p+j\,p+i})
+E_{i\,p+j}+E_{p+j\,i}$\\
$1\leq i,j\leq p$, $i< j$&&${}+E_{j\,p+i}+E_{p+i\,j}$\\ \hline
$\alpha(H)=\pm(H_i+H_j)$&1&
$Z_{ij}^\pm=\pm(E_{ij}-E_{ji}-E_{p+i\,p+j}+E_{p+j\,p+i})
-(E_{i\,p+j}+E_{p+j\,i})$\\ 
$1\leq i,j\leq p$, $i<j$&&${}+E_{j\,p+i}+E_{p+i\,j}$\\ \hline
\end{tabular}
\end{center}
\caption{Restricted roots and associated root vectors}\label{X}
\end{table}

The positive roots can be chosen as $\alpha(H)=H_i\pm H_j$, $1\leq
i<j\leq p$ and $\alpha(H)=H_i$, $i=1$, \dots, $p$. 
We therefore have the positive Weyl chamber
\begin{align*}
\a^+=\{H\in\a\colon ~H_1>H_2>\dots>H_p>0\}.
\end{align*}

The simple roots are given by $\alpha_i(H)=H_i-H_{i+1}$,
$i=1$, \dots, $p-1$ and $\alpha_p(H)=H_p$.

{\bf The action of the Weyl group.} The elements of the Weyl group $W$ act as permutations of the
 diagonal entries of ${\cal D}_X$ with eventual sign changes of any number of these
 entries.

The Lie algebra $\k$ is generated by the vectors $X_\alpha+ \theta X_\alpha$. We will use the notation
\begin{align*}
k^t_{X_\alpha}=e^{t(X_\alpha+ \theta X_\alpha)}.
\end{align*}

The linear space $\p$ has a basis formed by 
 $A_i\in\a$, $ 1\leq i\leq p$ and by the symmetric matrices
$X_\alpha^s:= \frac12( X_\alpha-\theta X_\alpha)$ which have the
 following form
\begin{align*}
X_{ir}:&= E_{i,2p+r}+E_{2p+r,i}; 1\leq i\leq p,\quad\quad 1\leq r\leq q-p; \\
Y_{ij}:&= E_{i,p+j}+ E_{j,p+i}+E_{p+j,i}+ E_{p+i,j},\quad 1\leq i<j\leq p; \\
Z_{ij}:&= E_{i,p+j}-E_{j,p+i}+E_{p+j,i} -E_{p+i,j},\quad 1\leq i<j\leq p.
 \end{align*}
If we followed the notation of \cite{PGPS_Lie2010}, 
we should write $(X_{ir}^+)^s$, etc.\ but we simplify the notation 
to $X_{ir}$, $Y_{ij}$ and $Z_{ij}$. 
If we write a matrix from the space $\p$ in the form
\begin{align*}
\left [\begin{array}{rr} 
0&B\\ 
B^T&0 \end{array}\right] = \left [\begin{array}{rrr} 
0&B_1&B_2\\ 
B_1^T&0&0\\
B_2^T&0&0\end{array}\right] 
\end{align*}
where $B_1$ is a square $p\times p$ matrix and $B_2$ is a $p\times
(q-p)$
matrix, then the matrices 
\begin{align*}
\left [\begin{array}{rrr} 
0&B_1& 0\\ 
B_1^T&0&0\\
0&0&0\end{array}\right]
\end{align*}
are generated by the vectors $A_i$ (for the diagonal entries of $B_1$
and $B_1^T$), $Y_{ij}$ and $Z_{ij}$ (for the non-diagonal entries),
whereas the matrices 
\begin{align*}
\left [\begin{array}{rrr} 
0&0&B_2\\ 
0&0&0\\
B_2^T&0&0\end{array}\right]
\end{align*}
are spanned by the vectors $X_{ir}$.

We now recall the useful matrix $S\in \SO(p+q)$ which allows us to
diagonalize simultaneously all the elements of $\a$.
Let
\begin{align*}
S=\left[\begin{array}{ccc}
\frac{\sqrt{2}}{2}\,I_p&0_{p\times (q-p)}&\frac{\sqrt{2}}{2}\,J_p\\
\frac{\sqrt{2}}{2}\,I_p&0_{p\times (q-p)}&-\frac{\sqrt{2}}{2}\,J_p\\
0_{(q-p)\times p}&I_{q-p}&0_{(q-p)\times p}
\end{array}
\right]
\end{align*}
where $J_p= (\delta_{i,p+1-i})$ is a matrix of size $p\times p.$
If $H=\left[\begin{array}{ccc}
0&{\cal D}_H&0\\
{\cal D}_H&0&0\\
0&0&0\end{array}\right]$ with ${\cal
D}_H=\diag[H_1,\dots,H_p]$ then 
$S^T\,H\,S=\diag[H_1,\dots,H_p,\overbrace{0,\dots,0}^{q-p},-H_p,
\dots,-H_1]$.

The ``group'' version of this result is as follows:
\begin{align*}
S^T\,e^H\,S=
\diag[e^{H_1},\dots,e^{H_p},\overbrace{1,\dots,1}^{q-p},e^{-H_p},
\dots,e^{-H_1}].
\end{align*}

\begin{remark}\label{S}
The Cartan projection $a(g)$ on the group ${\bf SO}_0(p,q)$,
defined as usual by
\begin{align*}
g=k_1 e^{a(g)} k_2,\ \ \ a(g)\in \overline{\a^+}, k_1,\ k_2\in K
\end{align*}
is related to the singular values of $g\in {\bf SO}(p,q)$ in the following way.
Recall that the singular values of $g$ are defined as the non-negative
square roots
of the eigenvalues of $g^Tg$. Let us write $H=a(g)$. We have

\begin{align*}
 g^T\,g &=
 k_2^T\,e^{2\,H}\,k_2 
= (k_2^T\,S)\,(S^T\,e^{2\,H}\,S)\,(S^T\,k_2)
\end{align*}
where $S^T\,e^{2\,H}\,S$ is a diagonal matrix with nonzero entries
$e^{2\,H_1}$, \dots, $e^{2\,H_p}$, $\overbrace{1,\dots, 1}^{q-p}$,
$e^{-2\,H_p}$, \dots, $e^{-2\,H_1}$. Let us write $a_j=e^{H_j}$.
Thus 
the set of $p+q$ singular values of $g$ contains the value 1
repeated $q-p$ times and the $2\,p$ values $a_1 $, \ldots, $a_p $,
$a_1^{-1}$, \ldots, $a_p^{-1}$.

Hence, in order to determine $a(g)$, we can compute the $p+q$ singular values 
of $g^Tg$ and omit $q-p$ values 1 always appearing among them. The 
$2p$ remaining singular values may be ordered
 $a_1\geq \ldots \geq a_p\geq a_p^{-1}\geq \ldots \geq a_1^{-1} $
 with $a_1\geq \ldots \geq a_p\geq 1$. Then
 \begin{align*} 
 a(g)=\left[\begin{array}{ccc}
 0&{\cal D}_{a(g)}&0\\
 {\cal D}_{a(g)}&0&0\\
 0&0&0
 \end{array}\right]
 \ \ {\rm with}\ \ 
 {\cal D}_{a(g)} =\diag[\log a_1,\dots,\log a_p]
 \end{align*}
 Summarizing, if for $g\in {\bf SO}(p,q)\subset {\bf SL}(p+q)$ the ${\bf SL}(p+q)$-Cartan decomposition writes $g=k_1 e^{\tilde a(g)} k_2$, $k_1,k_2\in {\bf SO}(p+q)$,
 then ${\cal D}_{a(g)} =\pi_p(\tilde a(g))$, where $\pi_p$ denotes the projection $\pi_p(\diag[h_1,\ldots,h_{p+q}])=\diag[h_1,\ldots,h_p]$.
\end{remark}

{\bf Singular elements of $\a$.} In what follows, we will consider
singular elements $X$, $Y\in \partial\a^+$. As
in \cite{PGPS_Lie2010}, we need to control the irregularity of $X$ and
$Y$, i.e. consider the simple positive roots annihilating $X$ and $Y$. A
special attention must be paid to the last simple root $\alpha_p$,
different from the roots $\alpha_i$, $i=1,\ldots,p-1$, that generate a
root subsystem of type $A_{p-1}$. We introduce the following
definition of the configuration of $X\in \overline{\a^+}$.
\begin{definition}
Let $X\in \overline{\a^+}$. 
There exist nonnegative integers $s_1\geq1$, \ldots, $s_r\geq1$, $u\geq0$ such that
\begin{align*}
{\cal D}_X = \diag [
\ \overbrace{x_{1},\dots,x_{1}}^{ {s}_{1}},
\overbrace{x_{2},\dots,x_{2}}^{ {s}_{2}},\dots,
\overbrace{x_{r},\dots,x_{r}}^{ {s}_{r}},
\overbrace {0,\dots, 0}^{ u}\ ]
\end{align*}
with $x_1>x_2> \ldots> x_r>0$ and $\sum s_i + u=p$. We say that
$[s_1,\dots,s_r;u]$ is the
{\bf configuration} of $X$. Writing ${\bf s}=(s_1,\ldots, s_r)$,
we will shorten the notation of the configuration of $X$ to $[{\bf
s};u]$. We will also write $X=X[{\bf s};u]$.
\end{definition}

 Note that $X=0$ is equivalent to $u=p$
and has configuration $[0;p]$. A regular $X\in \a^+$ has the configuration 
$[{\bf 1};0]=[1^p;0]$.
We extend naturally the definition of configuration to any
$X\in\a$, whose configuration is defined as that of the projection
$\pi(X)$
of $X$ on $\overline{\a^+}$.

In what follows, we will write $\max {\bf s}=\max_i s_i$
and $\max({\bf s}, u)= \max(\max {\bf s},u)$ . 
 We will show that in the case of the symmetric spaces
 $\SO_0(p,q)/{\bf SO}(p)\times {\bf SO}(q)$, $q>p$, the criterion for
 the existence of the density of the convolution
 $\delta_{e^X}^\natural \star \delta_{e^Y}^\natural$ is given by the
following definition of an eligible pair $X$ and $Y$:
\begin{definition}\label{defEligible}
Let $X=X[{\bf s};u]$ and $Y=Y[{\bf t};v]$ be two elements of $\a$. We say
that $X$ and $Y$ are eligible if 
\begin{align*}
\max({\bf s}, 2u) + \max({\bf t}, 2v) \le 2p.
\end{align*}
\end{definition}
Observe that if $X$ and $Y$ are eligible, then $X\not= 0$ and $Y\not= 0$.

\section{Necessity of the eligibility condition}
In the proof of the necessity of the eligibility condition 
we will use the result stated in \cite[Step 1, page 1767]{PGPS_Fun2010}:

\begin{lemma}\label{repeat}
Let $U=\diag([\overbrace{u_0,\dots,u_0}^r,u_1,\dots,u_{N-r}]$
and $V=\diag([\overbrace{v_0,\dots,v_0}^{N-s},v_1,\dots,v_s]$. where
$s+1\leq r<N$, $s\geq 1$, and the
$u_i$'s and $v_j$'s are arbitrary. 
Then each element of 
 $ \tilde
a(e^U\,\SU(N,\F)\,e^V)
 $
has at least $r-s$ entries equal to $u_0+v_0$.
\end{lemma}

We will use Lemma \ref{repeat} with $N=p+q$ in the proofs of Proposition \ref{Necessary} and Theorem
\ref{*power}.

\begin{proposition}\label{Necessary}
If $X[{\bf s};u]$ and $Y[{\bf t},v]$ are not eligible then 
the measure $\mu_{X,Y}$ is not absolutely continuous with respect to
the Lebesgue measure on $\a$.
\end{proposition}

\begin{proof}
Suppose 
$\max({\bf s},2\,u)+\max({\bf t},2\,v)> 2\,p$ and consider the matrices $a(e^X\,k\,e^Y)$, $k\in {\bf SO}(p)\times {\bf SO}(q)$. 
Applying Remark \ref{S}, the diagonal $p\times p$ matrix ${\cal D}_{a(e^X\,k\,e^Y)}$ contains the $p$ biggest
diagonal terms of the matrix
\begin{align*}
\tilde a(e^Xke^Y)=\tilde a(
\overbrace{(S^T\,e^{X}\,S)}^{e^{S^T\,X\,S}}
\,\overbrace{(S^T\,k\,S)}^{\in\SO(p+q)}\,
\overbrace{(S^T\,e^{Y}\,S)}^{e^{S^T\,Y\,S}})
\end{align*}

 If $u+v>p$ then there are 
$r-s=r+(N-s)-N=(2\,u+q-p)+(2\,v+q-p)-(p+q)=2\,(u+v-p)+(q-p)$
repetitions of $0+0=0$ in coefficients of $ \tilde a(e^X k e^Y)$. Therefore 0 occurs at least $u+v-p>0$ times as a diagonal entry of ${\cal D}_H$ for every $H\in a(e^X\,K\,e^Y)$ which implies that $a(e^X\,K\,e^Y)$ has empty
interior.

 If $2\,u+\max({\bf t})>2\,p$ denote $t=\max({\bf t})$. Let $Y_i\not=0$ be
 repeated $t$ times in ${\cal D}_Y$. Then there are 
$r-s=r+(N-s)-N=(2\,u+q-p)+t-(p+q)=2\,u+t-2p$ repetitions of $Y_i+0$ in coefficients of $ \tilde a(e^X k e^Y)$. Therefore $Y_i$ occurs at least $2\,u+t-2p>0$ times as a diagonal entry of ${\cal D}_H$ for every $H\in a(e^X\,K\,e^Y)$ 
which implies that $a(e^X\,K\,e^Y)$ has empty interior. 
\end{proof}

\section{Sufficiency of the eligibility condition}

We use basic ideas and some results and notations of \cite[Section
 3]{PGPS_Lie2010}.

 \begin{proposition}
 (i) The
density of the measure $m_{X,Y}$ exists if and only if its support
$Ke^X Ke^Y K$ has nonempty interior.\\
(ii) Consider the analytic map $T\colon K \times K \times K \to\SO_0(p,q)$
defined by
\begin{align*}
T(k_1,k_2,k_3)=k_1\,e^X\,k_2\,e^Y\,k_3.
\end{align*}
 If
 the derivative of $T$ is
surjective for some choice of ${\bf k}=(k_1,k_2,k_3)$, then 
 the set $T(K\times K\times K)=Ke^X Ke^Y K$ contains an open set.
 \end{proposition}
 \begin{proof}
 Part (i) follows from arguments explained in \cite{PGPS2} in the case
 of the support of the measure $\mu_{X,Y}$, equal to $a(e^X\,Ke^Y).$
 Part (ii) is justified for example in \cite[p. 479]{Helga}. 
 \end{proof}
 
 \begin {proposition}\label{U}
 Let $U_Z=\k+\Ad(Z)\k$. If there exists $k\in K$ such that 
 \begin {equation}\label{UXY}
 U_{-X}+\Ad(k)U_Y=\g
 \end {equation}
 then the measure $m_{X,Y}$ is absolutely continuous.
 \end{proposition}
\begin{proof}
 We want to show that the derivative of $T$ is
surjective for some choice of ${\bf k}=(k_1,k_2,k_3)$. 

Let $A,B,C\in\k$. The derivative of $T$ at ${\bf k}$ in the
 direction of
 $(A,B,C)$ equals
\begin{eqnarray}
dT_{{\bf k}}(A,B,C)&=&
\frac{d}{dt}\big|_{t=0}\ e^{tA}k_1\,e^X\,e^{tB}k_2\,e^Y\,e^{tC}k_3
\nonumber\\
&=&A\,k_1\,e^X\,k_2\,e^Y\,k_3
+k_1\,e^X\,B\,k_2\,e^Y\,k_3
+k_1\,e^X\,k_2\,e^Y\,C\,k_3\label{A}
\end{eqnarray}
We now transform the space of all matrices of the form (\ref{A})
without modifying its dimension: 
\begin{align*}
&\dim\{A\,k_1\,e^X\,k_2\,e^Y\,k_3
+k_1\,e^X\,B\,k_2\,e^Y\,k_3
+k_1\,e^X\,k_2\,e^Y\,C\,k_3\colon A,B,C \in\k\}\\
=&\, \dim
\{k_1^{-1}\,A\,k_1\,e^X\,k_2\,e^Y
+e^X\,B\,k_2\,e^Y
+e^X\,k_2\,e^Y\,C\colon A,B,C \in\k\}\\
=&\, \dim
\{A\,e^X\,k_2\,e^Y
+e^X\,B\,k_2\,e^Y
+e^X\,k_2\,e^Y\,C\colon A,B,C \in\k\}\\
=&\, \dim
\{e^{-X}\,A\,e^X+B+k_2\,e^Y\,C\,e^{-Y}k_2^{-1}\colon A,B,C
\in\k\}
 \end{align*}
 The space in the last line equals 
$\k+\Ad(e^{-X})(\k)+\Ad(k_2)\,(\Ad(e^Y)(\k)) = U_{-X}+\Ad(k_2)U_Y$.
\end{proof}

 In order to apply the condition (\ref{UXY}), we will consider 
 convenient root vectors and their 
 symmetrizations.
For $Z\in\a$, we define the space 
\begin{align*}
V_Z={\rm span}\{ X_\alpha^s\ |\ \alpha(Z)\not=0\},
\end{align*}
where $X_\alpha^s=X_\alpha-\theta X_\alpha$. Note that this space would be called $V_Z^S$ in the notation of 
\cite{PGPS_Lie2010}.
\begin{lemma}\label{3.2[7]}
 Let $Z\in\a$. The vector space $U_Z=\k+\Ad(e^Z)(\k)$ contains the root
vectors $X_\alpha$ for which $\alpha(Z)\not=0$. \A Consequently, $V_Z=V_{-Z}\subset U_{\pm Z}$.\B
\end{lemma}
\begin{proof}
\A Suppose $\alpha$ is a root such that $\alpha(Z)\not=0$.\B  Note that
$[Z,X_\alpha]=\alpha(Z)\,X_\alpha$ and
$[Z,\theta(X_\alpha)]=-\alpha(Z)\,\theta(X_\alpha)$ . Let $U=X_\alpha+\theta(X_\alpha)\in\k$. Now,
\begin{align*}
\Ad(e^Z)\, U
&=e^{\ad Z}\,(X_\alpha+\theta(X_\alpha))
=\sum_{k=0}^\infty\,\frac{(\ad Z)^k}{k!}\,(X_\alpha+\theta(X_\alpha))
\\
&=\sum_{k=0}^\infty\,\frac{(\ad Z)^k}{k!}\,X_\alpha
+\sum_{k=0}^\infty\,\frac{(\ad Z)^k}{k!}\,\theta(X_\alpha) \\
&=\sum_{k=0}^\infty\,\frac{(\alpha(Z))^k}{k!}\,X_\alpha
+\sum_{k=0}^\infty\,\frac{(-1)^k\,(\alpha(Z))^k}{k!}\,\theta(X_\alpha)
\\
&=e^{\alpha(Z)}\,X_\alpha+e^{-\alpha(Z)}\,\theta(X_\alpha).
\end{align*}
Therefore $X_\alpha=(e^{\alpha(Z)}-e^{-\alpha(Z)})^{-1}\left(-e^{-\alpha(Z)}\, U
+ \Ad(e^Z)\, U\right)\in \k+\Ad(e^Z)(\k)=U_Z$. The vector $\theta X_\alpha$ is a root vector for the root $-\alpha$, so we also have $\theta X_\alpha\in U_Z$. 
\end{proof}
\A
\begin{proposition}\label{VXVY}
If there exists $k\in K$ such that
\begin{equation}\label{conditionSym}
V_X+\Ad(k)\,V_Y= \p
\end{equation}
then the measure $m_{X,Y}$ is absolutely continuous.
\end{proposition}
\begin{proof}
We want to prove formula (\ref{UXY}). By Lemma \ref{3.2[7]}, we know that 
$V_X=V_{-X}\subset U_{-X}$ and $V_Y\subset U_Y$.
As $\k\subset U_X$, we see that the formula (\ref{conditionSym}) implies (\ref{UXY}).
\end{proof}
\B

Later in this section,(Theorem \ref{MainReduced}), we will show that the hypotheses of the last Proposition are always satisfied for $X$ and $Y$ eligible. For technical reasons, in order to make an induction proof work, we will show more, i.e. that a ``better'' matrix $k\in K$
exists such that the formula (\ref{conditionSym}) holds. The meaning of a ``better'' $k$ will be similar to the notion of a total matrix given in Definition \ref{Total}. Here is a definition and a lemma about total matrices in $K$. The reasonning of the proof of this lemma will be used in a more general setting in Steps 2 and 3 of the proof of Theorem \ref{MainReduced}. 

\begin{definition}\label{Total}
We say that a square $n\times n$ matrix $k$ is total if by removing any
$r<n$ rows and $r$ columns of $k$ we always obtain a nonsingular matrix.
\end{definition}

Note that this definition of totality is more restrictive than in
\cite[Definition 3.7]{PGPS_Lie2010}.

\begin{lemma}\label{total}
The set of matrices in $\SO(n)$ which are total is dense and open in
$\SO(n)$.
\end{lemma}

\begin{proof}
Consider first the set $M_{I,J}=
M_{\{i_1,\dots,i_j\},\{j_1,\dots,j_r\}}\subset
\SO(n)$ of orthogonal matrices which remain nonsingular once the rows of indices $i_1,\dots,i_j $ and
the columns of indices $j_1,\dots,j_r $ are removed. To see that such matrices
exist, take the identity matrix (whose determinant is 1 if we remove, say,
the first $r$ rows and columns). By taking convenient permutations of
the rows and columns of the identity matrix, we obtain an element
of $M_{I,J}$. Given that $\SO(n)\backslash M_{I,J}$ corresponds to
the set of zeros of a certain determinant function, it must be closed
and nowhere dense in $\SO(n)$. 

To conclude, it suffices to notice that the set of total matrices in
$\SO(n)$ is the finite intersection of all the sets $M_{I,J}$. 
\end{proof}

In the proof of the main Theorem \ref{MainReduced} we will need the following technical
lemma.
\begin{lemma}\label{calculs}
(i) For the root vectors $X_{ir}^+$, $Z_{ij}^+$, $Y_{ij}^+$, we have
\begin{align*}
\Ad(e^{t\,(X_{ir}^++\theta\,X_{ir}^+)})(X_{ir}
) &=\cos(2\,t)\,X_{ir}+2\,\sin(2\,t)\,A_i,\\
\Ad(e^{t\,(Y_{ij}^++\theta\,Y_{ij}^+)})(Y_{ij}) &=\cos(4\,t)\,Y_{ij}
+2\,\sin(4\,t)\,(A_i-A_j),\\
\Ad(e^{t\,(Z_{ij^+}+\theta\,Z_{ij}^+)})(Z_{ij}) &=\cos(4\,t)\,Z_{ij}
+2\,\sin(4\,t)\,(A_i+A_j).\\
\end{align*}
 (ii) The functions $\Ad(e^{t\,(X_{ir}^++\theta\,X_{ir}^+)})$,
$\Ad(e^{t\,(Y_{ij}^++\theta\,Y_{ij}^+)})$ and $\Ad(e^{t\,(Z_{ij}^++\theta\,Z_{ij}^+)})$ 
applied to the other symmetrized
root vectors do not produce any components in $\a$. 
\end{lemma}

\begin{proof}
It is just a matter of carefully evaluating 
\begin{align*}
\Ad(e^{t\,(Z+\theta\,Z)})(W)=e^{t\,\ad(Z+\theta\,Z)}(W)
=\sum_{k=0}^\infty\,(\ad(Z+\theta\,Z))^k(W)\,\frac{t^k}{k!}.
\end{align*}
For (ii), use the well known properties of the root system: $[\g_\alpha, \g_\beta]\in \g_{\alpha+\beta}$ and $[X_\alpha, \theta X_\alpha]\in\a$.
\end{proof}

By Proposition \ref{VXVY}, in order
to justify the sufficiency of the eligibility condition,
it is enough to prove the following theorem. This is the main result of this section.
\begin{theorem}\label{MainReduced}
 Let $G = \SO_0(p,q)$ and let $X$, $Y\in \a$. If $X$ and $Y$ are eligible then there exists a
 matrix $k\in K$ such that
\begin{equation}\label{VXVYp} 
V_X+\Ad(k)\,V_Y= \p.
\end{equation}
\end{theorem}

\begin{proof}
We will assume that $X=X[{\bf s};u]$ and $Y=Y[{\bf t};v]$.
Observe that the spaces $V_X$ and $V_Y$ depend on the Weyl chambers where $X$ and $Y$ belong.
However, see \cite[Lemma 3.3 and Reduction 1, p. 759]{PGPS_Lie2010}, the property (\ref{VXVYp}) is equivalent to $V_{w_1X}+\Ad(k')\,V_{w_2Y}= \p$ for any $w_1,w_2\in W$ and a convenient $k'\in K$. Throughout the proof we will assume that the diagonal entries of ${\cal D}_X$ and ${\cal D}_Y$ are non-negative and we will arrange (permute) them conveniently. 
 
To lighten the notation, for a matrix $c$ of size $p\times q$, we will consider the $(p+q)\times(p+q)$ symmetric matrix 
\begin{align*}
c^s=\left[ \begin{array}{cc}
0 & c\\
 c^T& 0
\end{array}
 \right]\in\p.
\end{align*}

The proof will be organized in the following way:
 
\begin{enumerate} 
\item Proof for $q=p+1$ using induction on $p$
\begin{enumerate}
\item Proof for $p=2$ and $q=3$
\item Proof of the induction step

\begin{enumerate}
\item Proof in the case $u>0$ or $v>0$

\item Proof in the case $X[p;0],Y[p;0]$
\end{enumerate}
\end{enumerate}
\item Proof that the case $(p,q)$ implies the case $(p,q+1)$.
\end{enumerate}

\bigskip\noindent {\bf 1. Proof for $q=p+1$ using induction on $p$}

\medskip\noindent {\bf (a) Proof for $p=2$ and $q=3$} 

This corresponds to the space $\SO_0(2,3)$. Only two configurations $[2;0]$ and $[1;1]$ may be realized by singular non-zero $X$ and $Y$.
When $Z\in{\overline \a^+}$, we have ${\cal D}_{Z[1;1]}=\diag[z,0]$, and ${\cal D}_{Z[2;0]}=\diag[z,z]$, $z\not=0$.
It is easy to check that in all 3 possible cases:
\begin{align*}
(i)\ X[2;0], Y[2;0]\ \ \ (ii) X[2;0], Y[1;1]\ {\rm or} X[1;1], Y[2;0]
 \ \ (iii) X[1;1], Y[1;1],
\end{align*}
$X$ and $Y$ are eligible. Note that 
\begin{align*}
\p&=\left\{\
\left[\begin{array}{ccc}
h_1&a &b\\
 c&h_2&d\end{array}\right]^s
\colon h_1,h_2,a,b,c,d\in\R\right\},\\
V_{Z[2;0]}&=\left\{\
\left[\begin{array}{ccc}
0&a&b\\
-a&0&c\end{array}\right]^s
\colon a,b,c\in\R\right\},\\
V_{Z[1;1]}&=\left\{\
\left[\begin{array}{ccc}
0&a&c\\
b&0&0\end{array}\right]^s
\colon a,b,c\in\R\right\}.
\end{align*}

If $k_1=
\left[\begin{array}{ccccc}
\sqrt{2}/2&-\sqrt{2}/2&0&0&0\\
\sqrt{2}/2&\sqrt{2}/2&0&0&0\\
0&0&1&0&0\\
0&0&0&\sqrt{2}/2&-\sqrt{2}/2\\
0&0&0&\sqrt{2}/2&\sqrt{2}/2\\
\end{array}\right]$ then 
\begin{align*}
\Ad(k_1)\,V_{Z[2;0]}&=\left\{\
\left[\begin{array}{ccc}
\sqrt{2}\,a/2&(a-b+c)/2&(a+b-c)/2\\
-\sqrt{2}\,a/2&(a-b-c)/2&(a+b+c)/2\end{array}\right]^s
\colon a,b,c\in\R\right\},
\end{align*}
and 
\begin{align*}
\Ad(k_1)\,V_{Z[1;1]}&=\left\{\
\left[\begin{array}{ccc}
-\sqrt{2}\,b/2&(a-c)/2&(a+c)/2\\
\sqrt{2}\,b/2&(a-c)/2&(a+c)/2\end{array}\right]^s
\colon a,b,c\in\R\right\}
\end{align*}

If 
$k_2=
\left[\begin{array}{ccccc}
\sqrt{2}/2&-\sqrt{2}/2&0&0&0\\
\sqrt{2}/2&\sqrt{2}/2&0&0&0\\
0&0&\sqrt{2}/2&0&-\sqrt{2}/2\\
0&0&0&1&0\\
0&0&\sqrt{2}/2&0&\sqrt{2}/2\\
\end{array}\right]$ then 
\begin{align*}
\Ad(k_2)\,V_{Z[1;1]}&=\left\{\
\left[\begin{array}{ccc}
-(b+c)/2&\sqrt{2}\,a/2&(-b+c)/2\\
(b-c)/2&\sqrt{2}\,a/2&(b+c)/2
\end{array}\right]^s
\colon a,b,c\in\R\right\}.
\end{align*}
 We verify easily that in the
cases (i)
and (iii)
we have $V_X+\Ad(k_1)\,V_Y=\p$. For $X[2;0]$ and $Y[1;1]$, we can see that
$V_X+\Ad(k_2)\,V_Y=\p$. 

\medskip\noindent {\bf (b) Proof of the induction step}

\smallskip\noindent{\bf (i) Proof in the case $u>0$ or $v>0$}

We consider the space $\SO_0(p,p+1)/\SO(p)\times\SO(p+1)$ with $p>2$ and the case when
$X[{\bf s};u]$ and $Y[{\bf t};v]$ are such that $u>0$ or $v>0$. We
assume $u\geq v$. We choose the predecessors in $\SO_0(p-1,p)/\SO(p-1)\times\SO(p)$ in the following way: 
\begin{align*}
X'=X'[{\bf s};u-1],\ Y'=Y'[{\bf t}';v]
\end{align*}
where ${\bf t}'$ means that we supress one term from the longest
block of size $\max{\bf t}$. Note that if $p>2$ then ${\bf t}'$ is
not the zero partition (otherwise, ${\bf t}$ would have been the
partition $[1]$ meaning that $u\geq v =p-1$ which would make $X$
and $Y$ ineligible).

We arrange $X$, $X'$, $Y$, $Y'$ in the following way.\\
1. The first diagonal
entry of ${\cal D}_X$ is zero and all the other zeros are at the
end. The diagonal entries of ${\cal D}_{X'}$ are those of ${\cal D}_X$ without the first
 zero:
\begin{align*}
{\cal D}_X=\diag[0,\overbrace{x_1,\ldots,x_{p-u}}^{\not=0},\overbrace{0,\ldots,0}^{u-1}],\ \ {\cal D}_{X'}=\diag[\overbrace{x_1,\ldots,x_{p-u}}^{\not=0},\overbrace{0,\ldots,0}^{u-1}]
\end{align*}

2. We put a longest block of size $t$ of equal diagonal entries $y_1$ of ${\cal D}_Y$
in the beginning of $Y$. The diagonal entries of
${\cal D}_{Y'}$ are those of ${\cal D}_Y$ with the first entry omitted:
\begin{align*}
{\cal D}_Y=\diag[\overbrace{y_1,\ldots,y_1}^{t},y_2,\ldots,y_s],\ \ {\cal D}_{Y'}= \diag[\overbrace{y_1,\ldots,y_1}^{t-1},y_2,\ldots,y_s].
\end{align*}

It is easy to check that if $X$, $Y$ are eligible in $\SO_0(p,p+1)$ then
$X'$, $Y'$ are eligible in $\SO_0(p-1,p)$. 

\medskip

\noindent{\bf Step 1.} By the induction hypothesis, there exists a
matrix $k_0\in \SO(p-1)\times\SO(p)$ such that
\begin{align}\label{VX'VY'p'}
V_{X'} +\Ad(k_0) V_{Y'}=\p'.
\end{align}
We embed $K'= \SO(p-1)\times\SO(p)$ in $\SO(p)\times\SO(p+1)$ in the
following way
\begin{align*}
K'=
\left[ 
\begin{array}{cccc}
1& & & \\
& \SO(p-1) &&\\
 &&1& \\
 &&& \SO(p)
\end{array} \right] \subset
\left[ 
\begin{array}{cc}
 \SO(p) &\\
 & \SO(p+1)
\end{array} \right].
\end{align*}

Hence, we have (taking the natural embedding of $\p'$ into $\p$)
\begin{align}
V_1:= V_{X'} +\Ad(k_0) V_{Y'}=\p'=\left[ \begin{array}{cc}
0 & B'\\
 B'^T& 0
\end{array}
 \right]\label{p}
\end{align}
where $
B'=\left [\begin{array}{c|c}0_{1\times 1}&0_{1\times p}\\ \hline
0_{p\times 1}&B''_{(p-1)\times p}
\end{array}\right]$ and $B''$ is arbitrary (note that $\p'$ is of dimension $(p-1)p$).
We must show that for some $k\in K$, the space $V_X+\Ad(k)\,V_Y= \p$,
i.e. that
\begin{itemize}
\item[(i)]
 $V_X+\Ad(k)\,V_Y$ contains $\p'$ embedded into $\p$ as in (\ref{p}).

\item[(ii)] $V_X+\Ad(k)\,V_Y$ contains all the matrices of the form
\begin{align*}
C= \left [\begin{array}{c|c}*&*\cdots*\\ \hline
*&\\
\vdots&0_{(p-1)\times p}\\
*&\\
\end{array}\right]^s.
\end{align*}
\end{itemize}

{\it New vectors in $V_X$ and $V_Y$.} 
In order to prove the induction conclusion, we must now use the
elements of $V_X$ and $V_Y$ which do not come from $ V_{X'}$ or $
V_{Y'}$. They appear by the interaction of, respectively, the
first diagonal entry of ${\cal D}_X$ with the 
others of ${\cal D}_X$ and the interaction of the first entry of ${\cal
 D}_Y$ with the others of ${\cal D}_Y$.
We see that the new independent root vectors in $V_X$ and
$V_Y$ are respectively
\begin{align*}
N_X=\{ Y_{1j},Z_{1j}, j=2, \dots, p+1-u \},\ \ \ \
N_Y=\{ X_1, Y_{1i}, i=t+1, \dots, p,\ Z_{1j}, j=2\dots, p\}
\end{align*}
where $t=\max{\bf t}>1$ and we wrote $X_1$ for $X_{11}$. Note that
$N_X$ has $2p-2u$ elements while $N_Y$ has $2\,p-t$.


\medskip

\noindent{\bf Step 2}. 
{\it We show that there exists $k'_0\in \SO(p-1)\times\SO(p)$ for which
(\ref{VX'VY'p'}) holds, and with the following property:\\

The space $V_2:=\Ad(k'_0){\rm span}(N_Y)$ is of dimension
 $2p-t$ and its elements can be written in the form 
\begin{align}
\left[\begin{array}{c|cccccc}0&\sigma_1&\dots&\sigma_r&
a_1&\dots&a_{p-r}\\ \hline
\tau_1& \\
\vdots&\\
\tau_{s}&&&&0\\
a_{p-r+1}&\\
\vdots\\
a_{2\,p-t}&
\end{array}\right]^s\label{ais}
\end{align}
with $r=[(t-1)/2]$, $s=t-1-r$, $a_i\in\R$ arbitrary,
$\sigma_i=\sigma_i(a_1,\dots,a_{2\,p-t})$ and 
$\tau_j=\tau_j(a_1,\dots,a_{2\,p-t})$, $i\le r, j\le s$.}\\
We will not need to write explicitely the functions $\sigma_i$ and $\tau_j$.
Note that $s=r$ if $t$ is
odd and $s=r+1$ if $t$ is even.

To justify Step 2, we write 
\begin{align*}
k_0=\left[ 
\begin{array}{cccc}
1& & & \\
& k_{01} &&\\
 &&1& \\
 &&& k_{02}
\end{array} \right]
\end{align*}
where $k_{01} \in \SO(p-1)$ and $k_{02} \in \SO(p)$. Let
$\alpha_1,\dots, \alpha_{p-1}$ be the columns of the matrix
$k_{01}$ and $\beta_1, \dots, \beta_p$ the columns of the matrix
$k_{02}$. A simple block multiplication to compute the
action of $\Ad(k_0)$ on the elements of $N_Y$ gives the linearly
independent matrices
 \begin{align}
\Ad(k_0)X_{1}=\left [\begin{array}{c|c}0& \beta_p^T\\ \hline
 0&0
\end{array}\right]^s,\ \
\Ad(k_0)Y_{1i}=\left [\begin{array}{c|c}0& \beta_{i-1}^T\\ \hline
 \alpha_{i-1}& 0
\end{array}\right]^s,\ i=t+1,\dots, p,\nonumber\\ 
Ad(k_0)Z_{1i}=\left
 [\begin{array}{c|c}0& \beta_{i-1}^T\\ \hline
-\alpha_{i-1}&0
\end{array}\right]^s, i=2,\dots, p.\label{with}
\end{align}

Let us write $\beta'_i$ for a column $\beta_i$ from which we have removed
the first $r$ entries and $\alpha_i'$ for a column $\alpha_i$ with the
first $s$ entries omitted. In order to prove the statement of
Step 2, we must show that the matrices obtained by replacing
$\beta_i$ by $\beta_i'$ and $\alpha_i$ by $\alpha_i'$ in (\ref{with})
are still linearly independent. This is equivalent to the
linear independence of the matrices
\begin{align}
\left [\begin{array}{c|c}
0& \beta_i'^T\\ \hline
-\alpha_i'& 0\\
\end{array}\right]^s,~ i=1,\dots, t-1,~~
 \left [\begin{array}{c|c}0& \beta_i'^T\\ \hline
 0& 0\\
\end{array}\right]^s,~ i=t,\dots, p,~~
 \left [\begin{array}{c|c}0&0\\ \hline
 \alpha_i'& 0\\
\end{array}\right]^s~ i=t,\dots, p-1.
\label{prime}
\end{align}
 We will reason in the same way as in Lemma \ref{total}.
 
It is enough to
show that there exists at least a choice of matrices $ k_{01}$ and $ k_{02}$
such that the matrices in (\ref{prime}) are linearly independent.
Then, as in Lemma \ref{total}, it will follow that
such matrices form a dense open subset in $\SO(p-1)\times \SO(p)$. By choosing $k_0'$
with the matrices in (\ref{prime}) linearly independent and 
close enough to $k_0$, the property (\ref{VX'VY'p'}) will be preserved for $k_0'$.

 Pick $k_{01}=I_{p-1}$ which implies
that $\alpha_i'={\bf 0}$ for $i=1$, \dots, $s$ and $\alpha_i'={\bf
 e}_{i-s}$ for $i>s$. With this choice, (\ref{prime}) becomes 
\begin{align*}
\left [\begin{array}{c|c}
0& \beta_i'^T\\ \hline
0& 0\\
\end{array}\right]^s,~ i=1,\dots, s,~~
\left [\begin{array}{c|c}
0& \beta_{i}'^T\\ \hline
{\bf - e}_{i-s}& 0\\
\end{array}\right]^s,~ i=s+1,\dots, t-1,~~
 \left [\begin{array}{c|c}0& \beta_i'^T\\ \hline
 0& 0\\
\end{array}\right]^s,~ i=t,\dots, p,~~
\\
 \left [\begin{array}{c|c}0&0\\ \hline
 {\bf e}_{i-s}& 0\\
\end{array}\right]^s~ i=t,\dots, p-1
\end{align*}
which are linearly independent provided that 
\begin{align*}
\left [\begin{array}{c|c}
0& \beta_i'^T\\ \hline
0& 0\\
\end{array}\right]^s,~ i=1,\dots, s,~~
 \left [\begin{array}{c|c}0& \beta_i'^T\\ \hline
 0& 0\\
\end{array}\right]^s,~ i=t,\dots, p
\end{align*}
are linearly independent. This is the case for a total matrix $k_{02}\in\SO(p)$
or by taking convenient permutations of
the rows and columns of the identity matrix $I_p$.

\medskip

\noindent{\bf Step 3.}{\it We show that there exists a proper subset
$N_X'$ of $N_X$ such that, if 
\begin{align*}
V_3:= {\rm span}(N_X') + V_{X'}+\Ad(k'_0) V_Y=
{\rm span}(N_X') + V_1+V_2
\end{align*}
then $\dim V_3= p\,q-1=\dim \p -1$ and $V_3$ is given by}
\begin{align}\label{V3general}
V_3= \left\{\left[\begin{array}{c|c}0&a_1~~a_2~~\dots~~a_p\\ \hline
a_{p+1}&\\
\vdots&\p'\\
a_{2p-1}
\end{array}\right]^s,\ \ a_1,\dots a_{2p-1} \in\R
\right\}.
\end{align}

Note that in matrices from the space $V_2$, there are $r=[(t-1)/2]$ pairs $(\sigma_i,\tau_i)$ plus possibly
an extra $\tau_s$ if $t$ is even and therefore $s=r+1$. Note also that
$t+2\,u\leq 2\,p$ implies that $p-u\geq s\geq r$. For $j\leq r\leq
p-u$, each pair
 $Y_{1j}=\left[\begin{array}{c|c}
0&{\bf e}_{j-1}^T\\ \hline
{\bf e}_{j-1}&0
\end{array}\right]^s$,
$Z_{1j}
=\left[\begin{array}{c|c}
0&{\bf e}_{j-1}^T\\ \hline
{-\bf e}_{j-1}&0
\end{array}\right]^s$ of elements of $N_X$ allows us to replace $\sigma_j$
and $\tau_j$ by independent variables. If $t$ is odd, all the
$\sigma_j$'s and $\tau_j$'s will be taken care off and at least 2 elements of $N_X$
will remain off $N'_{X}$. 
If $t$ is even, all the
$\sigma_j$ and $\tau_j$'s, $1\le j\le r$ will be replaced by independent variables and only
$\tau_s$ will remain. Now, letting the coefficient $a_1$
``vis-\`a-vis'' the remaining $\tau_s$ be equal to 1 and all the other
variables $a_i$ equal to 0,
either $\tau_s=1$ or $-1$ or $\tau_s\not=\pm 1$. If $\tau_s=1$ then $Z_{1s}$
allows us to introduce the missing independent variable, if
$\tau_s=-1$ then adding $Y_{1s}$ to $N'_{X}$ will do the trick. In the case $\tau_s\not=\pm 1$
we choose indifferently between $Y_{1s}$ and $Z_{1s}$.
 In all cases the set $N_X\setminus N'_{X}$
has at least one element.

\medskip

\noindent{\bf Step 4.} {\it Let $v_1$ be the positive root vector corresponding to an element of $N_X\setminus N'_{X}$. We denote $k_{1}^t= k^t_{v_1}$. There exists
$\epsilon>0$ such that for $t\in (0,\epsilon)$}
\begin{align*}
V_4^t:=\Ad(k_{1}^t) ( {\rm span}(N_X') +V_{X'}) + \Ad(k'_0) V_Y =V_3.
\end{align*}

Observe that $v_1$ is equal to $Z_{1j}^+$ or $Y_{1j}^+$ for 
one of the remaining $Z_{1j}$ or $Y_{1j}$ that was not used in the
preceding step. The space $V_4^t\subset V_4^0$ for all $t$ according to Lemma \ref{calculs}~(ii). 

Recall the definition of $k^t_{X_\alpha}=e^{t(X_\alpha+ \theta
X_\alpha)}$, $t>0$.
Let $d(t)=\dim V_4^t$; for $t=0$ we have $k_{1}^0=Id$, and
$\Ad(k_{1}^0) ( {\rm span}(N_X') +V_{X'}) + \Ad(k'_0) V_Y=V_3$
is of dimension $pq-1$, so $d(0)=pq-1$. 
The equality $d(t)=d(0)$ is equivalent to non-nullity of an
appropriate determinant continuous in $t$. Thus $d(t)=pq-1$ holds for
$t\in(0,\epsilon)$ for some $\epsilon>0$. As $V_4^t\subset V_4^0$,
the statement of Step 4 follows.

\medskip

\noindent{\bf Step 5.} {\it Generation of $A_1$}. By Lemma
\ref{calculs}, we have $\Ad(k_{1}^t)v_{1}^s= a_t v_{1}^s+b_t A_1 + c_t
A_{j}$ with $j\not=1$ and $b_t\not=0$ for $t\in(0,\epsilon)$ with $\epsilon$ small
enough. Consequently
\begin{align*} 
\Ad(k_{1}^t){\rm span}(v_{1}^s) + V_4^t=\p .
\end{align*}

\medskip

\noindent {\bf Conclusion.} We have $\p= \Ad(k_{1}^t)( \R v_1^s +{\rm span}N_X'
+ V_{X'}) + \Ad(k'_0) V_Y\subset \Ad(k_{1}^t)V_X + \Ad(k'_0) V_Y$,
so $ \Ad(k_{1}^t)V_X + \Ad(k'_0) V_Y=\p$. It follows that
\begin{align*}
V_X + \Ad((k_{1}^t)^{-1}k'_0) V_Y=\p.
\end{align*}

\smallskip\noindent{\bf (ii) Proof in the case $X[p;0],Y[p;0]$}

This case must be treated separately because the predecessors $X',Y'$ and consequently the sets 
$N_X$ and $N_Y$ are different than in case {\bf (i)}. The structure of the induction proof
is identical as in {\bf (i)}, with the Steps 2 and 3 executed together.

We choose both predecessors $X'[p-1;0],Y'[p-1;0]$ and arrange $X, X', Y, Y'$
in the same way as we did in the first part of the proof with $Y[{\bf t};v]$
and $Y'[{\bf t}';v]$.
In that case, $N_X=\{X_1,Z_{12},\dots, Z_{1p}\}=N_Y$ and the space $\Ad(k'_0)(N_Y)$
is generated by 
\begin{align}
\left [\begin{array}{c|c}
0& \beta_i^T\\ \hline 
-\alpha_i&0
\end{array}\right]^s,~i=1,\dots,p-1\ \ \hbox{and}~
\left [\begin{array}{c|c}
0& \beta_p^T\\ \hline 
0&0
\end{array}\right]^s.\label{S1}
\end{align}

 Recall that
\begin{align}
Z_{1j}=\left [\begin{array}{c|c}
0& {\bf e}_{j-1}^T\\ \hline 
-{\bf e}_{j-1}&0
\end{array}\right]^s,~j=2,\dots, p.\label{Z}
\end{align}
 We want to show that the matrices in (\ref{S1}) together with
those of (\ref{Z}) are linearly independent for a $k'_0\in\SO(p-1)\times\SO(p)$
for which the equality (\ref{VX'VY'p'}) holds. 
Note that if $k'_0=\left [\begin{array}{rr}-I_{p-1}&0\\0&I_{p} \end{array}
\right]$
($p$ odd) or $k'_0=\left [\begin{array}{rr}I_{p-1}&0\\0&-I_{p}
\end{array} \right]$ ($p$ even) then the matrices (\ref{S1}) and (\ref{Z}) are linearly
independent. Using once more the reasonning in Lemma \ref{total}, this implies
that the set of matrices $k'_0$ for which this is true, is open and
dense in $\SO(p-1)\times\SO(p)$.

 We conclude that if $N_X'=N_X\backslash \{X_1\}$ then ${\rm span}\,(N_X' + V_{X'}) + \Ad(k'_0) V_Y$ has the form
given in (\ref{V3general}). 

We reproduce the previous Step 4 and Step 5 using $v_1=X_1^+$.
The rest follows.
\medskip


\bigskip \noindent {\bf 2. Proof that the case $(p,q)$ implies the case
 $(p,q+1)$}
 
 \A We will show by induction that for any $q>p$, there exists a matrix $k\in K$
 such that (\ref{VXVYp}) holds.\B We know by the first part of the proof
 that this is true for $\SO_0(p,p+1)$.

Assume that $X$ and $Y$ are eligible in $\SO_0(p,q+1)$. Their configurations are
 eligible in $\SO_0(p,q)$. We write $X'$, $Y'$ when we work in $\SO_0(p,q)$.

We embed $K'= \SO(p)\times\SO(q)$ in $K=\SO(p)\times\SO(q+1)$ in
the following way
\begin{align*}
K'=
\left[ 
\begin{array}{ccc}
 \SO(p) &&\\
& \SO(q) &\\
&&1
\end{array} \right] \subset
\left[ 
\begin{array}{cc}
 \SO(p) &\\
 & \SO(q+1)
\end{array} \right].
\end{align*}
 The space $\p'$ is formed by matrices 
 \begin{align*}
\left [\begin{array}{rr} 
0&B\\ 
B^T&0 \end{array}\right], 
 \end{align*}
 where $B$ are $p\times q$ matrices.
We embed $\p'$
in $\p$ by adding a last column of zeros to $B$.

\medskip\noindent{\bf Step 1.} We suppose that there exists a matrix $k_0\in K'$ such that
\begin{align}\label{2Step1}
V_{X'} +\Ad(k_0) V_{Y'}=\p'.
\end{align}

Then, by \cite[Lemma 3.3]{PGPS_Lie2010}, for any permutations $s_1$ and $s_2$
of the diagonal entries of ${\cal D}_{X'}$ and ${\cal D}_{Y'}$, there
exists $k_0\in K'$ such that
\begin{align*}
V_{s_1 X'} +\Ad(k_0) V_{s_2 Y'}=\p'
\end{align*}
so we can permute the elements of $X'$ and $Y'$ in a convenient way
and still have the equality (\ref{2Step1}). We will arrange them in the following way
(where the stars denote nonzero entries): 
\begin{align*}
{\cal D}_{ X'}=\diag[\overbrace{0,\dots,0}^u,\star,\dots, \star],\ \ \ 
{\cal D}_{ Y'}=\diag[\star,\dots,
\star,\overbrace{0,\dots,0}^v].
\end{align*}
Let us denote $k_{01}\in\SO(p)$ and $k_{02}\in\SO(q)$ the matrices compositing $k_0$
corresponding in (\ref{2Step1}) to such $X'$ and $Y'$.
We can suppose that the matrix $k_{01}$ is total.

By the eligibility of $X$ and $Y$, $u+v\leq p$, so no two zeros in ${\cal D}_{X'} $ and ${\cal D}_{Y'} $ are
at the same position. 

Let $N=\{ X_{i,q+1} \}_{i=1}^p$. We set
$N_X:=V_X\cap N= \{X_{u+1,q+1},\dots,X_{p,q+1}\},
N_Y:=V_Y\cap N= \{X_{1,q+1},\dots,X_{p-v,q+1}\}$. We have $p-v\ge u$. 

\medskip\noindent{\bf Step 2.} 
Let 
\begin{align*}
k_1=\left[ 
\begin{array}{ccc}
 k_{01} &&\\
& k_{02} &\\
&&1
\end{array} \right] 
\end{align*}
where $k_{01}\in \SO(p)$ and $k_{02}\in \SO(q)$ are the blocks compositing
$k_0$. We then have
\begin{align}
V_{ X'} +\Ad(k_1) V_{ Y'}=\left [\begin{array}{cc}\p'
&{\bf 0}\end{array}\right]^s.\label{c} 
\end{align}

The space $ V_X + \Ad(k_1)V_Y$ contains, in addition to the matrices in (\ref{c}), the linear span of $ N_X + \Ad(k_1)N_Y$.

Denote the columns of the matrix $k_{01}$ by ${\bf c}_1,\dots,{\bf c}_p$.
By block multiplication in $\SO_0(p,q+1)$, we obtain
\begin{align*}
\Ad(k_1) X_{j,q+1}=
\left [\begin{array}{cc}0_{p\times q}&{\bf c}_j\end{array}\right]^s.
\end{align*}

This implies that the linear span of $ N_X + \Ad(k_1)N_Y $ contains the
following symmetric matrices:
\begin{align*}
\left [\begin{array}{cc}0_{p\times q}&{\bf c}_1\end{array}\right]^s,
\dots,
\left [\begin{array}{cc}0_{p\times q}&{\bf
 c}_{u}\end{array}\right]^s,
 \left [\begin{array}{cc}0_{p\times q}&{\bf
 e}_{u+1}\end{array}\right]^s,
\dots, \left [\begin{array}{cc}0_{p\times q}&{\bf
 e}_p\end{array}\right]^s,
\end{align*}
which are linearly independent by the totality of $k_{01}$. Thus
 $V_X + \Ad(k_1)V_Y= \p.$
\end{proof}

\A

We conclude this section with an example to illustrate our proof.
\begin{example}
Consider $X=X[2;1],Y=Y[1,1;1]$ in $\so(3,4)$. We write $X$ and $Y$ in such a way that  ${\cal D}_X=\diag[0,a,a]$ and ${\cal D}_Y=\diag[b,c,0]$. Their predecessors in $\so(2,3)$ are $X'$ and $Y'$ such that ${\cal D}_{X'}=\diag[a,a]$ and ${\cal D}_{Y'}=\diag[c,0]$.

Note that $X$ and $Y$ form an eligible pair and so are $X'=X[2;0]$ and $Y'=Y'[1;1]$.
In Step 1, we show that there exists a matrix $k_0=\left[\begin{array}{cccc}1&0&0&0\\0&k_{0,1}&0&0\\
0&0&1&0\\0&0&0&k_{02}\end{array}\right]$ with $k_{01}\in \SO(2)$ and $k_{0,2}\in\SO(3)$ such that 
\begin{align*}
V_{X'}+\Ad(k_0)\,V_{Y'}=\left[\begin{array}{cccc}
0&0&0&0\\
0&*&*&*\\
0&*&*&*
\end{array}\right]^s
\end{align*}
where $*$ designates an arbitrary element. We have
\begin{align*}
N_X=\{Z_{12},Y_{12},Z_{13},Y_{13} \},\ \ \ \
N_Y=\{ X_1,Z_{12},Y_{12},Z_{13},Y_{13}\}.
\end{align*}
In Step 2, we observe that
\begin{align*}
\Ad(k_0){\rm span}(N_Y)&=\left\{\left[\begin{array}{c|c}
0&a_1~~a_2~~a_3\\ \hline
a_4&\\
a_5&0
\end{array}\right]^s,\ \ a_1,\dots a_6\in\R
\right\}
\end{align*}
since the matrices 
\begin{align*}
\left [\begin{array}{c|c}0& \beta_1^T\\ \hline
 -\alpha_1&0\\
\end{array}\right]^s,\ 
\left [\begin{array}{c|c}0& \beta_1^T\\ \hline
\alpha_1&0\\
\end{array}\right]^s,\ 
\left [\begin{array}{c|c}0& \beta_2^T\\ \hline
-\alpha_2&0\\
\end{array}\right]^s,\
\left [\begin{array}{c|c}0& \beta_2^T\\ \hline
\alpha_2&0\\
\end{array}\right]^s,\ 
 \left [\begin{array}{c|c}0& \beta_3^T\\ \hline
 0&0\\
\end{array}\right]^s
\end{align*}
are linearly independent. Note that in this case, there are no $\sigma_i$ and no $\tau_i$.
 
Now, $V_X=\hbox{span}\overbrace{\{Z_{12},Y_{12},Z_{13},Y_{13}\}}^{N_X}\cup V_{X'}$ while
$V_Y=\hbox{span}\,\overbrace{\{X_1,Z_{12},Z_{13},Z_{14},Y_{13}\}}^{N_Y}
\cup 
V_{Y'}$.   We can show that 
\begin{align*}
\Ad(e^{t\,(Z_{1,2}^++\theta\,Z_{1,2}^+)})(V_{X'})
+\Ad(k_0)\,(\hbox{span}\,N_Y\cup V_{Y'})\\
&=\left [\begin{array}{cccc}
0&*&*&*\\
*&*&*&*\\
*&*&*&*
\end{array} \right]^s
\end{align*}
for $t$ small enough (with $t$ small enough, the dimension will not decrease). Now,

\begin{align*}
\Ad(e^{t\,(Z_{1,2}^++\theta\,Z_{1,2}^+)})
\overbrace{(\hbox{span}\{Z_{12}\} \cup V_{X'})}^{\subset V_X}
+\Ad(k_0)\,(\overbrace{\hbox{span}\,N_Y\cup V_{Y'}}^{V_Y})
&=\left [\begin{array}{cccc}
*&*&*&*\\
*&*&*&*\\
*&*&*&*
\end{array} \right]^s=\p
\end{align*}
if $t$ close to 0 
since $\Ad(e^{t\,(Z_{12}^++\theta\,Z_{12}^+)})(Z_{12}) 
=\cos(4\,t)\,Z_{12}
+2\,\sin(4\,t)\,(A_1+A_2)$.  Therefore,
\begin{align*}
V_X
+\Ad(\overbrace{e^{-t\,(Z_{1,2}^++\theta\,Z_{1,2}^+)}\,k_0}^k)\,V_Y
&=\Ad(e^{-t\,(Z_{1,2}^++\theta\,Z_{1,2}^+)})\p=\p
\end{align*}
which means that the density exists.

\end{example}

\B
\section{Applications}

\A We now extend our results to the complex and quaternion cases.\B

Recall that ${\bf SU}(p,q)$ is the subgroup of ${\bf SL}(p+q,\C)$ such that $g^*\,I_{p,q}\,g=I_{p,q}$ 
while 
${\bf Sp}(p,q)$ is the subgroup of ${\bf SL}(p+q,\H)$ such that $g^*\,I_{p,q}\,g=I_{p,q}$. Their respective maximal compact subgroups are ${\bf S}({\bf U}(p)\times {\bf U}(q))$ and 
${\bf Sp}(p)\times {\bf Sp}(q)\equiv {\bf SU}(p,\H)\times {\bf SU}(q,\H)$.

Their subspaces $\p$ can be described as $\left[\begin{array}{cc}0&B\\B^*&0\end{array}\right]$ where $B$ is an arbitrary complex (respectively quaternionic) matrix of size $p\times q$. The Cartan subalgebra $\a$ is chosen in the same way as for $\so(p,q)$.

\begin{cor}\label{iff}
Consider the symmetric spaces ${\bf SO}_0(p,q)/{\bf
SO}(p)\times{\bf SO}(q)$, ${\bf SU}(p,q)/{\bf S}({\bf
U}(p)\times{\bf U}(q))$ and ${\bf Sp}(p,q)/{\bf
Sp}(p)\times{\bf Sp}(q)$, $q>p$.

Let $X$, $Y\in\a$. Then the measure $\delta_{e^X}^\natural \star
\delta_{e^Y}^\natural$ is absolutely continuous
 if and only if $X$ and $Y$ are eligible, as defined in Definition \ref{defEligible}.
\end{cor}

\begin{proof}
Let $X$, $Y\in\a$. If they are eligible then since 
\begin{align*}
a(e^X\,({\bf SO}(p)\times{\bf SO}(q))\,e^Y)\subset a(e^X\,{\bf S}({\bf
U}(p)\times{\bf U}(q))\,e^Y)\subset a(e^X\,({\bf
Sp}(p)\times{\bf Sp}(q))\,e^Y),
\end{align*}
it follows from Theorem \ref{MainReduced} that these sets have nonempty interior. Hence the density exists in all three cases.

On the other hand, given Lemma \ref{repeat}, one can reproduce Proposition \ref{Necessary} using $\F=\C$ and $\F=\H$ to show that the eligibility condition is necessary in the complex and quaternionic cases. 
\end{proof}

We will conclude this paper with two further applications.

\begin{proposition}
Let $X$ and $Y\in\a$ be such that
$\left(\delta_{e^X}^\natural\right)^{*2}$ and 
$\left(\delta_{e^Y}^\natural\right)^{*2}$ are absolutely continuous. Then 
$\delta_{e^X}^\natural*\delta_{e^Y}^\natural$ is absolutely continuous.
\end{proposition}

\begin{proof}
\A Let $X=X[{\bf s};u]$ and $Y=Y[{\bf t};v]$.\B
We know that the couple $(X,X)$ is eligible; therefore
\begin{align*}
2\,\max\{{\bf s},2\,u\}\leq 2\,p.
\end{align*}

In the same manner, $\max\{{\bf t},2\,v\}\leq p$.
Hence,
\begin{align*}
\max\{{\bf s},2\,u\}+\max\{{\bf t},2\,v\}\leq p+p=2\,p
\end{align*}
which means that $X$ and $Y$ are eligible. Consequently,
$\delta_{e^X}^\natural*\delta_{e^Y}^\natural$ is absolutely continuous.
\end{proof}

If $X\in\a$ and $X\not=0$, it is important to know for which convolution powers $l$
the measure $\left(\delta_{e^X}^\natural\right)^{l}$ is absolutely continuous.
 This problem is equivalent to the study of the absolute continuity of convolution powers of uniform orbital measures $\delta_{g}^\natural=m_K*\delta_{g}*m_K$
 for $g\not\in K$.

It was proved in \cite[Corollary 7]{PGPS_Fun2010} that it is always the case for $l\ge r+1$, where
$r$ is the rank of the symmetric space $G/K$. It was also conjectured 
 (\cite[Conjecture 10]{PGPS_Fun2010}) that $r+1$ is optimal for this property,
 which was effectively proved for symmetric spaces of type $A_n$(\cite[Corollary 18]{PGPS_Fun2010}). 
 In the following theorem, the conjecture is shown not to hold on symmetric spaces
 of type \A $B_p$\B, where $r=p$. Thanks to the rich structure of the root system $B_p$,
 already all $p$-th powers of orbital measures are absolutely continuous and $p$ is optimal 
 for this property.

\begin{theorem}\label{*power}
On symmetric spaces ${\bf SO}_0(p,q)/{\bf
SO}(p)\times{\bf SO}(q)$, ${\bf SU}(p,q)/{\bf S}({\bf
U}(p)\times{\bf U}(q))$ and ${\bf Sp}(p,q)/{\bf
Sp}(p)\times{\bf Sp}(q)$, $q>p$, 
for every nonzero $X\in\a$, the measure $(\delta_{e^X}^\natural)^p$ is absolutely
continuous. Moreover, $p$ is the smallest value for which this is
true: if $X$ has a configuration $[1;p-1]$ then the measure $(\delta_{e^X}^\natural)^{p-1}$
is singular.
\end{theorem}

\begin{proof}
We will write $S_X^l$ for the set $a(e^X\,K\,\dots,K\,e^X)$ where the
factor $e^X$ appears $l$ times. Note that $(\delta_{e^X}^\natural)^l$ is
absolutely continuous if and only if $S_X^l$ has nonempty interior.

We prove first that for $l<p$, the measure $(\delta_{e^X}^\natural)^l$ may not be
absolutely continuous. Let $X=X[1;p-1]$. Using Lemma \ref{repeat}
repeatedly, as in the proof of Proposition \ref{Necessary}, we show that for $l< p$, there are at least $p-l$
diagonal entries of ${\cal D}_H$ which are equal to 0 for every $H\in
 S_X^l$. Consequently, $S_X^l$ has empty interior and
 $(\delta_{e^X}^\natural)^l$ is not absolutely continuous when $l\le p-1$.

We will now show that $(\delta_{e^X}^\natural)^p$ has a density for every $X\not=0$. 

If $X=X[{\bf s};0]$ then the measure $(\delta_{e^X}^\natural)^2$ is already absolutely continuous (the couple $(X,X)$
 is eligible). Suppose then that $X=X[{\bf s};u]\in\overline{\a^+}$, $u>0$. 

Remark that if $H\in S_X^l$ then $a(e^X\,K\,e^H)\subset S_X^{l+1}$.
Indeed, we have $e^X\,k_1\,E^X\,\dots\,k_{l-1}\,e^X=k_a\,e^H\,k_b$ and
therefore $a(e^X\,K\,e^H)=a(e^X\,K\,k_a\,e^H\,k_b)
=a(e^X\,K\,e^X\,k_1\,\dots\,k_{l-1}\,e^X)
\subset S_X^{l+1}$.

We claim that there exists $H\in S_X^{p-1}$ such that $H=H[{
 1}^{p-1};1]$ or $H\in \a^+$.

We prove the claim using induction on $p$. If $p=2$ then
$S_X^{p-1}=\{X\}$ and the result follows (in that case, $u$ cannot be
higher than 1 for $X\not=0$). 

Suppose that the claim is true for $p-1\geq 2$. Let $K_0=\left 
[\begin{array}{cccc} \SO(p-1)&0&0&0\\
0&1&0&0\\ 0&0&\SO(q-1)&0\\ 0&0&0&1\end{array}\right]$. Consider 
the set $B=a(e^X\,K_0\,e^X\,\dots\,e^X)$ with $p-1$ factors $e^X$. By the
induction hypothesis, there exists $H_0\in B$ with
$H_0=H_0[{ 1}^{p-2};2]$ or $H_0=H_0[{ 1}^{p-1};1]$. In the second case,
we are done. 

 If $H_0=H_0[{ 1}^{p-2};2]\in B$, we can assume that the diagonal entries which are 0 in ${\cal
 D}_X$ and in $ H_0$ are at the end. 
 We note that $X$ and $H_0$ considered without their last entries are
 eligible in $\SO_0(p-1,q-1)$, their
 configurations being $[{\bf s};u-1]$ and $[{ 1}^{p-2};1]$ respectively. Hence
 $a(e^X\,K_0\,e^{H_0})$ has nonempty interior in the subspace
 $\overline{\a^+}\cap\{H_p=0\}$. Therefore, there exists $H\in
 a(e^X\,K_0\,e^{H_0}) \subset S_X^{p-1}$ with $H=H[{ 1}^{p-1};1]$
 which proves the claim. 

To conclude, we take $H\in S_X^{p-1}$ with $H\in\a^+$ or $H=H[{1}^{p-1};1]$. In both cases, $X$ and $H$ are eligible, so by Corollary \ref{iff}
the set $a(e^X\,K\,e^H)$ has nonempty interior. 
 As $a(e^X\,K\,e^H) \subset S_X^p$, this ends the proof.
\end{proof}

\section{Conclusion}

With this paper and with \cite{PGPS_Lie2010}, we have now obtained sharp criteria on singular $X$ and $Y$ for the existence of the density of $\delta_{e^X}^\natural \star
\delta_{e^Y}^\natural$ for the root systems of type $A_n$ and type $B_p$. Thanks to \cite{PGPS_Fun2010} and Theorem \ref{*power} of the present paper, sharp criteria
are now given for the $l$-th convolution powers $ (\delta_{e^X}^\natural)^l$ to be absolutely continuous for any $X\not=0$, $X\in\a$. 

Although there is considerable similarity between the criteria for both type of spaces, a characterization of eligibility that would be applicable for all Riemannian symmetric spaces of noncompact type has yet to emerge. The solution of the second problem in Theorem \ref{*power} seems to indicate that the answer may depend on the type of the symmetric space.

\noindent{P.\ Graczyk\\
D\'epartement de math\'ematiques\\
Universit\'e d'Angers\\
piotr.graczyk@univ-angers.fr}\\[2mm]
{P.\ Sawyer\\
Department of Mathematics and Computer Science\\
Laurentian University\\
psawyer@laurentian.ca
}

\end{document}